\documentclass[12pt,twoside,draft]{article}

\usepackage{times,amsmath,amsthm,amssymb,bm}
\numberwithin{equation}{section}

\theoremstyle{plain}
\newtheorem{thm}{Theorem}[section]

\newtheorem{prop}[thm]{Proposition}

\theoremstyle{definition}

\newtheorem{conj}[thm]{Conjecture}

\newcommand{\memo}[1]{}

\begin{document}

\title{\bf{Effect of the time delay on the stability and instability of the logistic map}}
\author{{\normalsize by}\\[0.5ex] Yoshifumi Takenouchi and Yasushi Ota}
\date{\normalsize August 26, 2009}
\maketitle



\vspace*{-13ex}

\begin{abstract}\vspace*{0ex}
\begin{center}\hspace*{-1.5ex}\begin{minipage}{65ex}\hspace{2.5ex}
A proper discretization of the logistic differential equation, which is preserving these two distinct equilibrium solutions and their unstability and stability, suggest that we need to examine the time delay of the logistic map. According to Murray \cite{murray}, the effect of delay in models is ``usually" to increase the potential for instability. However the word ``usually" is really ambiguous.
In this paper, we mathematically formulate and prove the two conjectures about stability and instability.
\end{minipage}
\end{center}\vspace*{6ex}
\end{abstract}

\section{Introduction}

The logistic differential equation (cf. \cite{math-model}, etc \ldots ) is given by

\begin{equation}
\frac{dx}{dt}=rx\left(1-\frac{x}{K}\right),
\label{eq:logistic-1}
\end{equation}
where $r$ and $K$ stand for 
the reproduction rate and the carrying capacity for the population, respectively.
Thus, let us assume that $r$ is a positive parameter and $K$ is any fixed positive constant.
Here, we note that the equilibrium solution $X_{1}=0$ is unstable for any $r>0$, and 
the equilibrium solution $X_{2}=K$ is stable for any $r>0$.

Let us find a proper discretization of (\ref{eq:logistic-1}), which preserves these two distinct equilibrium solutions and their instability and stability.
First, we consider the following discretization:
\begin{equation}
\begin{array}{rccl}
&{\displaystyle \frac{x(t+\Delta t)-x(t)}{\Delta t}}
&=&{\displaystyle rx(t)\left(1-\frac{x(t)}{K}\right)} \hspace{2ex} \mbox{ as }\Delta t\to0\\
\mbox{} \\
\Longrightarrow&{\displaystyle \frac{X_{n+1}-X_{n}}{h}}
&=&{\displaystyle rX_{n}\left(1-\frac{X_{n}}{K}\right)}\\
&\mbox{(forward difference)}&&\\
\end{array}
\label{discretization-1}
\end{equation}
\noindent where $X_{n+1}:=x(t+\Delta t),X_{n}:=x(t)$, and $h:=\Delta t(>0)$.

\noindent This becomes a first order nonlinear difference equation:
\begin{equation}
X_{n+1}=(rh+1)X_{n}-\frac{rh}{K}X_{n}^2.
\label{eq:dis-1}
\end{equation}
We note that the fixed point $X=0$ is unstable for any $r>0$, and 
the fixed point $X=K$ is stable if $0<r<\frac{2}{h}$, and unstable if $\frac{2}{h}<r$.
Therefore, though the non-trivial equilibrium solution of (\ref{eq:logistic-1}) is stable for any $r>0$, the non-trivial fixed point of (\ref{eq:dis-1}) is stable only for $0<r<\frac{2}{h}$.
In this sense, (\ref{eq:dis-1}) is not a proper discretization of (\ref{eq:logistic-1}).

Next, we consider the following discretization:
\begin{equation}
\begin{array}{rccl}
&{\displaystyle \frac{x(t+\Delta t)-x(t)}{\Delta t}}
&=&{\displaystyle rx(t)\left(1-\frac{x(t+\Delta t)}{K}\right)}\mbox{ as }\Delta t\to0\\
\mbox{} \\
\Longrightarrow&{\displaystyle \frac{X_{n+1}-X_{n}}{h}}
&=&{\displaystyle rX_{n}\left(1-\frac{X_{n+1}}{K}\right)}\\
&\mbox{(forward difference)}&&\\
\end{array}
\label{discretization-4}
\end{equation}
\noindent where $X_{n+1}:=x(t+\Delta t),X_{n}:=x(t)$, and $h:=\Delta t(>0)$.

\noindent This becomes a first order nonlinear difference equation:
\begin{equation}
X_{n+1}=\frac{(1+rh)X_{n}}{{\displaystyle 1+\frac{rh}{K}X_{n}}}.
\label{eq:ver-4}
\end{equation}
We note that the fixed point $X=0$ is unstable for any $r>0$, and 
the fixed point $X=K$ is stable any $r>0$.
Then, the non-trivial fixed point of (\ref{eq:ver-4}) is stable for any $r>0$.
In this sense, we see that (\ref{eq:ver-4}) is a proper discretization of (\ref{eq:logistic-1}).

Observe that if the term $X_{n+1}$ on the right-hand side of (\ref{discretization-4}) 
is replaced by $X_{n}$, then 
the range of stability of the trivial fixed point does not change and  
the range of stability of the non-trivial fixed point shrinks 
from $(0,+\infty)$ to $(0,\frac{2}{h})$.

If we extend this observation, 
we can naturally guess that if the term $X_{n+1}$ on the right-hand side of (\ref{discretization-4}) 
is replaced by $X_{n-\tau}$, then 
the range of stability of the trivial fixed point does not change and  
the range of stability of the non-trivial fixed point shrinks further as $\tau$ increases.

According to Murray \cite{murray}, the effect of delay in models is ``usually" to increase the potential for instability. However the word ``usually" is really ambiguous.

In this paper, we mathematically formulate and prove the following conjectures:
for the discrete logistic equation which includes the time delay term $X_{n-\tau}$, 
$(i)$ the range of stability of the trivial fixed point does not change as we change the value of $\tau$; and 
$(ii)$ the range of stability of the non-trivial fixed point constricts gradually as we increase the value of $\tau$.

\section{Discrete Logistic Equation with Time Delay}

Equation (\ref{eq:logistic-1}) is not realistic in modeling a population with distinct
breeding seasons for it incorporates instantaneous response to changes in the environment.
In reality, this response to changes usually takes effect after some time delay or time lag.

In making things more realistic in modeling populations, we shall consider another equation which will be the main equation of this study. This equation is similar to the discrete logistic equation, however, it incorporates the time delay between increased deaths and the resulting decrease in population reproduction. This equation is known to be the Discrete Logistic Equation with Time Delay \cite{math-model}, which is given by
\begin{equation}
X_{n+1} = X_n + rX_n\biggl(1 - \displaystyle \frac{X_{n-\tau}}{K}\biggr)
\label{eq:time-delay1}
\end{equation}
Again, $\tau$ and K are the reproduction rate of the population and carrying capacity respectively. We denote the time delay to be  Equation (\ref{eq:time-delay1}), similar to Equation (\ref{eq:logistic-1}), has two distinct fixed points, namely the trivial fixed point denoted by $X_1 = \underbrace{(0, 0, \cdots, 0)}_{\tau+1}$, and the non-trivial fixed point given by $X_2 = \underbrace{(K, K, \cdots, K)}_{\tau+1}.$
Futhermore, we can guess that the stability of the population size of equation (\ref{eq:time-delay1}) also depends on the value of $\tau$ as equation (\ref{eq:logistic-1}). Equation (\ref{eq:time-delay1}) is applied by May on A.J. Nicholson's experimental data concerning Australian blowflies \cite{math-model}, a pest in the Australian sheep industry. May well-approximated the experimental data Nicholson had gathered using equation (\ref{eq:time-delay1}) and is shown by Figure 3.8.1 in \cite{math-model}. The parameter values used for the output of the model are $\tau = 0.106 days^{-1}, K = 2.8 \times 10^3$ flies and $\tau = 17$ days.
\section{Stability of the trivial fixed point}

In this chapter, we summarizes the results we obtained from the computations. 
\begin{table}[htbp]
\begin{center}
\begin{tabular}{|c|c|c|} \hline
Time Delay & Trivial fixed point & Range of stability \\\hline
$\tau = 0$ & $X_1 = 0$ & $ -2 < r < 0 $ \\\hline
$\tau = 1$ & $X_1 = (0, 0)$ & $ -2 < r < 0 $ \\\hline
$\tau = 2$ & $X_1 = (0, 0, 0)$ & $ -2 < r < 0 $ \\\hline
$\tau = 3$ & $X_1 = (0, 0, 0, 0)$ & $ -2 < r < 0 $ \\\hline
$\vdots$ & $\vdots$ & $\vdots$ \\\hline
\end{tabular}
\end{center}
\caption{Time Delay Stability of the trivial fixed point}
\label{fig:table1}
\end{table}
From the Table \ref{fig:table1} below, notice that as the value of $\tau$ increases, the range of stability of
the trivial fixed point is the same all throughout, which is $-2 < r < 0$. For the non-trivial
fixed point, we can observe that as the $\tau$ increases, its range of stability decreases. From the computations made, it shows numerically the statement of Murray that the effect of delay in models is usually to increase the potential for instability. Hence, the following proposition is formed from these observations.
\begin{prop}
Let
\begin{equation}
X_{n+1} = X_{n} + r X_{n}\Bigl(1-\displaystyle \frac{X_{n-\tau}}{K}\Bigr)
\end{equation}
\noindent for any $\tau$.
Then, there are always two fixed points  $X_1 = \underbrace{(0, 0, \cdots, 0)}_{\tau+1}, X_2 = \underbrace{(K,K, \cdots ,K)}_{\tau+1}$ and
$X_1$ is stable if \ $-2 < r < 0.$
\end{prop}
\begin{proof}
Given the dicrete logistic equation with time delay in this form
\begin{equation}
X_{n+1} = X_{n} + r X_{n}\Bigl(1-\displaystyle \frac{X_{n-\tau}}{K}\Bigr)
\end{equation}
there are always two fixed points, $X_1 = \underbrace{(0, 0, \cdots, 0)}_{\tau+1}, X_2 = \underbrace{(K,K, \cdots ,K)}_{\tau+1}.$

\noindent To get the range of stability of the trivial fixed point, we need its Jacobian matrix and it is
given by
\begin{equation}
J(X_1) = 
\left(
  \begin{array}{ccccccc}
      0 &  1  &  0  &  0  &  \cdots  &  \cdots  &  0  \\
      0 &  0  &  1  &  0  &  \cdots  &  \cdots  &  0    \\
      0 &  0  &  0  &  \ddots  &   &    &  \vdots  \\
      \vdots & \vdots   & \vdots   & \ddots   & \ddots   &    &  \vdots  \\
      \vdots & \vdots   & \vdots   &    & \ddots   & \ddots   &  \vdots  \\
      0 &  0  &  0  & \cdots  & \cdots  &  0  &  1  \\
      0 &  0  &  0  & \cdots  & \cdots  &  0  &  1+r  \\
  \end{array}
\right)_{(\tau+1)\times(\tau+1)}
\end{equation}
Since $J(X_1)$ is an upper triangular matrix, then the eigenvalues of it is just its diagonal
entries, $ \lambda_i = 0 \ (i = 1, 2, \cdots, n), \lambda_{\tau+1} = 1+r.$
Applying the stability test (cf. \cite{chaos}, etc \ldots ), the fixed point $X_1 = \underbrace{(0, 0, \cdots, 0)}_{\tau+1}$ is stable if $-2 < r < 0$.

Thus the proof is completed.
\end{proof}

\section{Stability of the Non-trivial fixed point}

We then use the revised Jury's Conditions for Stability Test in finding the stability of the discrete logistic equation with time delay, given by Equation (\ref{eq:time-delay1}), at its non-trivial fixed point.
Having computed for the range of stability of the equation when $\tau$ = 0, 1, 2, 3,   
we have the following results summarized by Table \ref{fig:table2}.
\begin{table}[htbp]
\begin{center}
\begin{tabular}{|c|c|c|} \hline
Time Delay & Non-trivial fixed point & Range of stability \\\hline
$\tau = 0$ & $X_2 = K$ & $ 0 < r < 2 $ \\\hline
$\tau = 1$ & $X_2 = (K, K)$ & $ 0 < r < 1 $ \\\hline
$\tau = 2$ & $X_2 = (K, K, K)$ & $ 0 < r < 0.618034 $ \\\hline
$\tau = 3$ & $X_2 = (K, K, K, K)$ & $ 0 < r < 0.445042 $ \\\hline
$\vdots$ & $\vdots$ & $\vdots$ \\\hline
\end{tabular}
\end{center}
\caption{Time Delay Non-trivial Fixed Point Range of Stability}
\label{fig:table2}
\end{table}
The results we have from our computations that is summarized by the table
above evidently show that as time delay increases, the range of stability of the equation determined by the value of $r$ shrinks. According to Murray \cite{murray}, the effect of delay in models is usually to increase the potential for instability. This means that as we increase the time delay of the population to respond to changes in the environment, the risk of instability increases. With such instability, tremendous changes is possible to occur to the population.
\begin{conj}
\textit{Let  
\begin{equation}
X_{n+1} = X_{n} + r X_{n}\Bigl(1-\displaystyle \frac{X_{n-\tau}}{K}\Bigr)
\end{equation}
for any $\tau$. Then there are always two fixed points  $X_1 = \underbrace{(0, 0, \cdots, 0)}_{\tau+1}, X_2 = \underbrace{(K,K, \cdots ,K)}_{\tau+1}$ and
$X_2$ is stable if \ $0 < r < f(\tau),$
where 
$$
f(\tau) : N \to \textbf{R} \ \mbox{is a monotone decreasing function of $\tau$.}
$$}
\end{conj}

\medskip

This conjecture generalizes the behavior of the stability of the discrete logistic equation with time delay at its non-trivial fixed point. 
The proof of the conjecture above is given by induction.

\begin{proof}
We begin the proof by assuming that $P^{(k)}_{n+1}(\lambda)$, a polynomial that satisfies $k+2$ Jury Condition. Then afterwards, we will show that the polynomial $P^{(k+1)}_{n+1}(\lambda)$
will be able to satisfy $k+3$ Jury conditions, that is, for any $n \ge k$.

Now, we define $P^{(k)}_{n+1}(\lambda)$ as 
\begin{equation}
\begin{split}
&P^{(k)}_{n+1}(\lambda) \\
&= a_0^{(k)}\lambda^{n+2-k} + a_1^{(k)}\lambda^{(n+2-k)-1} + \cdots + a_{(n+2-k)-2}^{(k)}\lambda^2 + a_{(n+2-k)-1}^{(k)}\lambda + a_{n+2-k}^{(k)},
\end{split}
\label{eq:Pol1}
\end{equation}
where $a_0^{(k)}, a_{(n+2-k)-1}^{(k)}, a_{n+2-k}^{(k)} \neq 0$ and others are $0$.

$P^{(k)}_{n+1}(\lambda)$ have a value satisfies the $k+2$ Jury Condition given by 
\begin{equation}
|a_{n+2-k}^{(k)}| > a_0^{(k)}
\end{equation}
which does not depend on any $n\ge k.$
Then, $P^{(k+1)}_{n+1}(\lambda)$ is given by 
\begin{equation}
\begin{split}
P^{(k+1)}_{n+1}(\lambda) = &a_0^{(k+1)}\lambda^{n+2-(k+1)} + a_1^{(k+1)}\lambda^{(n+2-(k+1))-1} + \cdots + a_{(n+2-(k+1))-2}^{(k+1)}\lambda^2 \\
&+ a_{(n+2-(k+1))-1}^{(k+1)}\lambda + a_{n+2-(k+1)}^{(k+1)} \\
= &a_0^{(k+1)}\lambda^{n+1-k} + a_1^{(k+1)}\lambda^{(n+1-k)-1} + \cdots + a_{(n+1-k)-2}^{(k+1)}\lambda^2 \\ 
&+ a_{(n+1-k)-1}^{(k+1)}\lambda + a_{n+1-k}^{(k+1)}
\end{split}
\label{eq:Pol2}
\end{equation}

Since the polynomial is the preceding polynomial after $k$, by the process of obtaining the coefficients described in (\ref{eq:juryformula}) or using Table \ref{fig:revisedjury} in Appendix, we must use the coefficients of the polynomial $P^{(k)}_{n+1}(\lambda)$ in order to find the coefficients of $P^{(k+1)}_{n+1}(\lambda)$.

Then the coefficients of $P^{(k+1)}_{n+1}(\lambda)$ are given by as follows:

\vspace{3ex}

$a_0^{(k+1)} = \left|
  \begin{array}{cc}
     a_{n+2-k}^{(k)}  &  a_{(n+2-k)-1}^{(k)}  \\
     a_0^{(k)}  &  a_1^{(k)}  \\
  \end{array}
\right|
= -a_{(n+2-k)-1}^{(k)}a_0^{(k)}
$

\vspace{5ex}

$a_1^{(k+1)} = \left|
  \begin{array}{cc}
     a_{n+2-k}^{(k)}  &  a_{(n+2-k)-2}^{(k)}  \\
     a_0^{(k)}  &  a_2^{(k)}  \\
  \end{array}
\right|
= 0
$

\vspace{5ex}

\hspace{5ex} \vdots

\vspace{5ex}

$a_{(n+1-k)-2}^{(k+1)} = \left|
  \begin{array}{cc}
     a_{n+2-k}^{(k)}  &  a_{2}^{(k)}  \\
     a_0^{(k)}  &  a_{(n+2-k)-2}^{(k)}  \\
  \end{array}
\right|
= 0
$

\vspace{5ex}

$a_{(n+1-k)-1}^{(k+1)} = \left|
  \begin{array}{cc}
     a_{n+2-k}^{(k)}  &  a_{1}^{(k)}  \\
     a_0^{(k)}  &  a_{(n+2-k)-1}^{(k)}  \\
  \end{array}
\right|
= a_{n+2-k}^{(k)} a_{(n+2-k)-1}^{(k)}
$

\vspace{5ex}

$a_{n+1-k}^{(k+1)} = \left|
  \begin{array}{cc}
     a_{n+2-k}^{(k)}  &  a_{0}^{(k)}  \\
     a_0^{(k)}  &  a_{n+2-k}^{(k)}  \\
  \end{array}
\right|
= (a_{n+2-k}^{(k)})^2 - (a_{0}^{(k)})^2
$.

Due to this, we say that the only coefficients of the polynomial $P^{(k+1)}_{n+1}(\lambda)$ which are not equal to zero are $a_0^{(k+1)}, a_{(n+1-k)-1}^{(k+1)}, a_{n+1-k}^{(k+1)}$. Determining the $k+3$ Jury Condition, we have
\begin{equation}
|a_{n+1-k}^{(k+1)}| > |a_0^{(k+1)}|
\end{equation}

Then, since this condition succeeds $k+2$ Jury Condition for any $n\ge k$, we can say that this condition futher shorten the range of stability of $r$. And thus, proving the conjecture.

\end{proof}

\section{Appendix}

In this appendix, we present Jury Conditions in the following form, which is essentialy same to Jury Conditons for Stability Test given in \cite{murray} (cf. Appendix B in \cite{murray}) .

Given a Jacobian matrix with dimension $(n+1)\times(n+1)$, a characteristic polynomial of order $n+1$ is of the form:
\begin{equation}
P(\lambda) = a_0 \lambda^{n+1} + a_1\lambda^n + a_2 \lambda^{n-1} + \cdots + a_n\lambda + a_{n+1}
\end{equation}
where $a_i \in \textbf{R}, \ i=0, 1, \cdots , n+1.$

From this, we have the characteristic polynomial of order $n+1$ given by:
\begin{equation}
P_{n+1}^{(1)}(\lambda) = a_0^{(1)} \lambda^{n+1} + a_1^{(1)}\lambda^n + a_2^{(1)} \lambda^{n-1} + \cdots + a_n^{(1)}\lambda + a_{n+1}^{(1)}
\end{equation}
where $a_i^{(1)} \in \textbf{R}, \ i=0, 1, \cdots , n+1, \hspace{2ex} a_{n+1}^{(1)} \neq 0.$

The formulas used in defining $b_k, c_k, \cdots, q_k$ are rewritten as follows:

\vspace{3ex}

\begin{equation}
\begin{split}
&a_k^{(2)} := \left|
  \begin{array}{cc}
     a_{n+1}^{(1)}  &  a_{n-k}^{(1)}  \\
     a_0^{(1)}  &  a_{k+1}^{(1)}  \\
  \end{array}
\right|
, \hspace{2ex} k = 0, 1, 2, \ldots, n \\
&\mbox{} \\
&a_k^{(3)} := \left|
  \begin{array}{cc}
     a_{n}^{(2)}  &  a_{n-1-k}^{(2)}  \\
     a_0^{(2)}  &  a_{k+1}^{(2)}  \\
  \end{array}
\right|
, \hspace{2ex} k = 0, 1, 2, \ldots, n-1 \\
&\mbox{} \\
&\hspace{15ex} \vdots  \\
&\mbox{} \\
&a_k^{(n+2)} := \left|
  \begin{array}{cc}
     a_{3}^{(n+1)}  &  a_{2-k}^{(n+1)}  \\
     a_0^{(n+1)}  &  a_{k+1}^{(n+1)}  \\
  \end{array}
\right|
, \hspace{2ex} k = 0, 1, 2 
\end{split}
\label{eq:juryformula}
\end{equation}

\noindent Thus, the Jury's stability table is given by:
\begin{table}[htbp]
\begin{center}
\begin{tabular}{|c|c|c|c|c|c|c|c|c|c|} \hline
row & $\lambda^0$ & $\lambda^1$ & $\lambda^2$ & $\lambda^3$ & $\cdots$ & $\lambda^{n-2}$ & $\lambda^{n-1}$ & $\lambda^{n}$ & $\lambda^{n+1}$ \\\hline
1 & $a_{n+1}^{(1)}$ & $a_{n}^{(1)}$ & $a_{n-1}^{(1)}$ & $a_{n-2}^{(1)}$ & $\cdots$ & $a_{3}^{(1)}$ & $a_{2}^{(1)}$ & $a_{1}^{(1)}$ & $a_{0}^{(1)}$ \\\hline
2 & $a_{0}^{(1)}$ & $a_{1}^{(1)}$ & $a_{2}^{(2)}$ & $a_{3}^{(1)}$ & $\cdots$ & $a_{n-2}^{(1)}$ & $a_{n-1}^{(1)}$ & $a_{n}^{(1)}$ & $a_{n+1}^{(1)}$ \\\hline
3 & $a_{n}^{(2)}$ & $a_{n-1}^{(2)}$ & $a_{n-2}^{(2)}$ & $a_{n-3}^{(2)}$ & $\cdots$ & $a_{2}^{(2)}$ & $a_{1}^{(2)}$ & $a_{0}^{(2)}$ &  \\\hline
4 & $a_{0}^{(2)}$ & $a_{1}^{(2)}$ & $a_{2}^{(2)}$ & $a_{3}^{(2)}$ & $\cdots$ & $a_{n-2}^{(2)}$ & $a_{n-1}^{(2)}$ & $a_{n}^{(2)}$ &  \\\hline
$\vdots$ & $\vdots$ & $\vdots$ & $\vdots$ & $\vdots$ & $\vdots$ & $\vdots$ & $\vdots$ &  &  \\\hline
2n-5 & $a_{3}^{(n+1)}$ & $a_{2}^{(n+1)}$ & $a_{1}^{(n+1)}$ & $a_{0}^{(n+1)}$ &  &  &  &  &  \\\hline
2n-4 & $a_{0}^{(n+1)}$ & $a_{1}^{(n+1)}$ & $a_{2}^{(n+1)}$ & $a_{3}^{(n+1)}$ &  &  &  &  &  \\\hline
2n-3 & $a_{2}^{(n+2)}$ & $a_{1}^{(n+2)}$ & $a_{0}^{(n+2)}$ &  &  &  &  &  &  \\\hline
\end{tabular}
\end{center}
\caption{Revised Jury's Stability Table}
\label{fig:revisedjury}
\end{table}

Finally, we rewrite the Jury Conditions for stability given as follows:
\begin{align*}
1. \hspace{2ex} & P^{(1)}_{n+1}(1) > 0  \\
2. \hspace{2ex} & (-1)^n P^{(-1)}_{n+1}(1) > 0 \\
3. \hspace{2ex} &|a_{n+1}^{(1)}| < a_0^{(1)} \\
4. \hspace{2ex} &|a_{n}^{(2)}| > a_0^{(2)} \\
5. \hspace{2ex} &|a_{n-1}^{(3)}| > a_0^{(3)} \\
\hspace{2ex}&\hspace{5ex} \vdots \\
(n+2).\hspace{2ex} &|a_{2}^{(n+2)}| > a_0^{(n+2)} 
\end{align*}
Also, we need to define polynomial in which the coefficients are from Table2, $a_{n+1}^(i) \ (i=2, 3, \ldots, n+2)$. Using the same notations we had in (2.3), we denote the polynomials as follows:

\begin{equation}
\begin{split}
P^{(2)}_{n+1}(\lambda) &= a_0^{(2)}\lambda^{n} + a_1^{(2)}\lambda^{n-1} + a_2^{(2)}\lambda^{n-2} + \cdots + a_{n-1}^{(2)}\lambda + a_{n}^{(2)}, \\
P^{(3)}_{n+1}(\lambda) &= a_0^{(3)}\lambda^{n-1} + a_1^{(3)}\lambda^{n-2} + a_2^{(3)}\lambda^{n-3} + \cdots + a_{n-2}^{(3)}\lambda + a_{n-1}^{(3)}, \\
\hspace{2ex}\hspace{15ex} \vdots  &\\
P^{(n+1)}_{n+1}(\lambda) &= a_0^{(n+1)}\lambda^{3} + a_1^{(n+1)}\lambda^{2} + a_2^{(n+1)}\lambda + a_{n}^{(n+1)},  \\
P^{(n+2)}_{n+1}(\lambda) &= a_0^{(n+2)}\lambda^{2} + a_1^{(n+2)}\lambda + a_{2}^{(n+2)}.
\label{eq:Pol3}
\end{split}
\end{equation}


\end{document}